\documentclass[10pt]{elsarticle}
\usepackage[cp1251]{inputenc}
\usepackage[english]{babel}
\usepackage{amsmath}
\usepackage{amssymb}
\usepackage{amsfonts}
\usepackage{mathtools}
\usepackage{dsfont}
\usepackage{rotating}
\usepackage{graphicx}
\usepackage{floatflt,epsfig}
\usepackage{lineno,hyperref}
\usepackage{enumerate}
\usepackage{colortbl}
\usepackage{array,tabularx,tabulary,booktabs}
\usepackage{longtable}
\usepackage{multirow}
\usepackage{wrapfig}
\usepackage{subcaption}
\usepackage[table]{xcolor}
\newcolumntype{^}{>{\currentrowstyle}}

\journal{Linear Algebra and its Applications}
\newtheorem{lemma}{Lemma}

\newtheorem{theorem}{Theorem}
\newtheorem{corollary}{Corollary}
\newtheorem{proposition}{Proposition}
\newtheorem{problem}{Problem}
\newcommand{\proof}{\medskip\noindent{\bf Proof.~}}
\DeclareMathOperator{\Spec}{Spec}

\DeclareMathOperator{\textif}{if}
\DeclareMathOperator{\otherwise}{otherwise}

\begin{document}
\renewcommand{\abstractname}{Abstract}
\renewcommand{\refname}{References}
\renewcommand{\tablename}{Figure.}
\renewcommand{\arraystretch}{0.9}
\thispagestyle{empty}
\sloppy

\begin{frontmatter}
\title{Integral graphs obtained by dual Seidel switching}

\author[01,02,03,04]{Sergey Goryainov}
\ead{sergey.goryainov3@gmail.com}

\author[04,05]{Elena~V.~Konstantinova}
\ead{e\_konsta@math.nsc.ru}

\author[06]{Honghai Li}
\ead{lhh@mail.ustc.edu.cn}

\author[01]{Da Zhao}
\ead{jasonzd@sjtu.edu.cn}

\address[01]{Shanghai Jiao Tong University, 800 Dongchuan RD. Minhang District \\ Shanghai 200240, China}
\address[02]{Krasovskii Institute of Mathematics and Mechanics, S. Kovalevskaja st. 16 \\ Yekaterinburg
620990, Russia}
\address[03] {Chelyabinsk State University, Brat'ev Kashirinyh st. 129\\Chelyabinsk  454021, Russia}
\address[04]{Sobolev Institute of Mathematics, Ak. Koptyug av. 4, Novosibirsk 630090, Russia}
\address[05]{Novosibirsk State University, Pirogova str. 2, Novosibirsk, 630090, Russia}
\address[06]{Jiangxi Normal University, Nanchang, Jiangxi 330022, China}


\begin{abstract}
Dual Seidel switching is a graph operation introduced by W.~Haemers in 1984. This operation can change the graph, however it does not change its bipartite double, and because of this, the operation leaves the squares of the eigenvalues invariant. Thus, if a graph is integral then it is still integral after dual Seidel switching. In this paper two new infinite families of integral graphs are obtained by applying dual Seidel switching to the Star graphs and the Odd graphs. In particular, three new $4$-regular integral graphs with their spectra are found. \\
\end{abstract}

\begin{keyword}
dual Seidel switching; multiple Kronecker covering; integral graph; 4-regular graph; Star graph; Odd graph
\vspace{\baselineskip}
\MSC[2010] 05C25\sep 05E10\sep 05E15
\end{keyword}
\end{frontmatter}

\section{Introduction}\label{sec1}
\subsection{Motivation}
For a graph $\Gamma$ with adjacency matrix $A(\Gamma)$, the {\it bipartite double} of $\Gamma$ is the graph with adjacency matrix $\begin{bmatrix} 0 & A(\Gamma) \\ A(\Gamma) & 0 \end{bmatrix}$. The bipartite double of $\Gamma$ can be seen as the Kronecker product of $\Gamma$
with $K_2$, which is also known as the {\it Kronecker covering} of $\Gamma$.
If two (or more) non-isomorphic graphs have isomorphic Kronecker coverings, we speak of
a multiple Kronecker covering (see \cite{IP08}).
 In this paper we establish a connection between the phenomenon of multiple Kronecker covering and the dual Seidel switching operation, which was introduced in \cite{H84}. We then apply dual Seidel switching to the Star graphs and to the Odd graphs, which gives two new infinite families of integral graphs. In particular, Theorem \ref{TStar} and Corollary \ref{invStar} give a new 4-regular graph with spectrum $\{ (-3)^{7}, (-2)^{13}, (-1)^3, 0^{15},  1^1, 2^{15}, 3^{5}, 4^1 \}$ and Theorem \ref{TOt} gives two new 4-regular graphs with spectra $\{(-3)^5, (-2)^4, (-1)^9, 1^5, 2^{10}, 3^1, 4^1\}$ and $\{ (-3)^4, (-2)^6, (-1)^8, 1^6, 2^8, 3^2, 4^1 \}$.  Also, Theorem \ref{TStar} and Corollary \ref{inv1Star} give a vertex-transitive 4-regular graph, which is already known from~\cite[page 402]{MW15} as the Cayley graph $H_{15}$. In this paper we found an infinite family of Cayley graphs that contains this graph and whose bipartite double is isomorphic to the Star graph.

\subsection{Integral graphs}

A graph is {\it integral} if all eigenvalues of its adjacency matrix are integers~\cite{BH12,HS74}. For a graph $\Gamma$, let $Spec(\Gamma)$ denote its spectrum. In this paper we deal with regular integral graphs. Some constructions of non-regular integral graphs can be found in~\cite{MT11,WH08}.

It was proved by D.~Cvetkovi\'{c}~\cite{C75} in $1975$ that the set of connected regular integral graphs of any fixed degree is finite. Classification of $3$-regular integral connected graphs was given in $1976$ by F.~C.~Bussemaker, D.~Cvetkovi\'{c}~\cite{BC76} and A.~J.~Schwenk~\cite{Sch78}. There are exactly $13$ non-isomorphic cubic integral connected graphs.

There is no complete classification of $4$-regular integral connected graphs. In what follows, we briefly observe known results on their classification.
In $1998$, D.~Cvetkovi\'{c}, S.~Simi\'{c}, and D.~Stevanovi\'{c}~\cite{CSS98} found $1888$ possible spectra of $4$-regular bipartite integral graphs, more than $500$ of which do not correspond to a graph as it was shown in~\cite{S98}. They also published a list of $65$ known $4$-regular connected integral graphs. All connected $4$-regular integral graphs that do not contain $\pm3$ in the spectrum were determined by D.~Stevanovi\'{c}~\cite{S03} in $2003$; he found 24 such graphs. In the same paper it was shown that the number of vertices in $4$-regular integral graphs satisfies $8 \leqslant n \leqslant 1260$, except for $5$ identified spectra. In $2007$, the upper bound for $n$ was improved~\cite{SAF07} so that now we have $8 \leqslant n \leqslant 560$, and the number of possible spectra of connected $4$-regular bipartite integral graphs was decreased down to $828$. Moreover, all $828$ feasible spectra of connected $4$-regular bipartite integral graphs, and all $47$ connected $4$-regular bipartite integral graphs with up to $24$ vertices were listed. The largest of these $828$ spectra has $560$ vertices, and actually there are only $12$ spectra with more than $360$ vertices.
The spectra of $14$ non-bipartite connected $4$-regular integral graphs found by decomposing the known $47$ bipartite connected $4$-regular integral graphs are given in~\cite{SAF07}.

In $2015$, M.~Minchenko and I.~M.~Wanless \cite{MW15} investigated $4$-regular integral Cayley graphs. It was shown that up to isomorphism, there are $32$ connected $4$-regular integral Cayley graphs ($17$ of them are bipartite) and $27$ connected $4$-regular integral arc-transitive graphs ($16$ of them are bipartite, and $16$ of them are Cayley graphs). Complete catalogues of these graphs are available by~\url{users.monash.edu.au/~iwanless/data/graphs/IntegralGraphs.html}.

\section{Preliminaries}
\subsection{Dual Seidel switching}
For any simple graph $\Gamma$ with adjacency matrix $A(\Gamma)$ and an order 2 automorphism $\varphi$ of $\Gamma$ interchanging only non-adjacent vertices,   we have
$$PA(\Gamma)P^T=A(\Gamma),$$
where $P$ is the permutation matrix corresponding to the automorphism $\varphi$. It is easy to verify that $PA(\Gamma)$ is a symmetric (0,1)-matrix with zero diagonal and thus can be viewed as an adjacency matrix of some simple graph. The resulting graph is said to be obtained from $\Gamma$ by \emph{dual Seidel switching} induced by $\varphi$ (\cite{H84} and \cite[Theorem 3.1]{EFHHH99}). Note further  that  $(PA(\Gamma))^2=(A(\Gamma))^2$. In particular, if $\Gamma$ is integral, then a graph obtained from $\Gamma$ by the dual Seidel switching is integral as well.

\begin{lemma}\label{Images} Let $\Gamma$ be a graph and $\varphi$ be an order 2 automorhism interchanging only non-adjacent vertices. Let $\Delta$ be the graph obtained from $\Gamma$ by dual Seidel switching induced by $\varphi$.
For the neighbourhood $\Delta(x)$ of a vertex $x$ of the graph $\Delta$, the following conditions hold:
$$
\Delta(x) =
\left\{
  \begin{array}{ll}
    \Gamma(x), & \hbox{if $\varphi(x) = x$;} \\
    \Gamma(\varphi(x)), & \hbox{if $\varphi(x) \ne x$.}
  \end{array}
\right.
$$
\end{lemma}
\begin{proof}
    It follows from the definition of dual Seidel switching. $\square$
\end{proof}

\subsection{Kronecker covering and combinatorial polarities}
The following lemma gives basic properties of Kronecker covering.

\begin{lemma}[{\cite[Proposition 1]{IP08}}]\label{connectedness}
Kronecker covers of graphs are bipartite. If $G$ is bipartite, then its Kronecker cover consists of two copies of $G$. If $G$ is connected and non-bipartite then its Kronecker cover is connected.
\end{lemma}

There is a criterion for a graph to be a Kronecker cover.
\begin{lemma}[{\cite[Proposition 2]{IP08}}]\label{Polarity}
Let $K$ be a bipartite graph with bipartition $(V_1,V_2)$ and let $\pi \in {\rm Aut}(K)$ be a fixed-point free involution such that $\pi$ interchanges the bipartition, i.e. $\pi(V_1) = V_2$. If the vertices $v$ and $\pi(v)$ are non-adjacent for every vertex $v$ of $K$, then $K$ is a Kronecker cover.
\end{lemma}

An order 2 automorhism $\pi$ satisfying the condition of Lemma \ref{Polarity} is called a \emph{combinatorial polarity} (see \cite{IP08}). Combinatorial polarities thus lead to quotient graphs where pairs of images are identified.

\subsection{Cayley graphs}
Let $G$ be a group and $S$ be an inverse-closed identity-free subset in $G$.
We define the \emph{left Cayley graph} $Cay_L(G,S)$ (resp. \emph{right Cayley graph} $Cay_R(G,S)$) as the graph whose vertices are the elements of the group $G$ and with two vertices $x,y$ being adjacent whenever $y^{-1}x \in S$ (resp. $xy^{-1} \in S$) holds. For an element $\pi \in G$, let
$\varphi^\ell_\pi$ and $\varphi^r_\pi$ denote the left and the right shifts by the element $\pi$.
Let $L_G = \{\varphi^\ell_\pi~|~\pi\in G\}$ and $R_G = \{\varphi^r_\pi~|~\pi\in G\}$ be the groups of left and right shifts $G$, respectively.

\begin{lemma}\label{aut} The following statements hold.\\
{\rm (1)} $L_G$ is a group of automorphisms of $Cay_L(G,S)$;\\
{\rm (2)} $R_G$ is a group of automorphisms of $Cay_R(G,S)$;\\
{\rm (3)} For an element $\pi \in G$, the mapping $\varphi^r_\pi$ is an automorphism of $Cay_L(G,S)$ if and only if $\pi S\pi^{-1} = S$ holds;\\
{\rm (4)} For an element $\pi \in G$, the mapping $\varphi^\ell_\pi$ is an automorphism of $Cay_R(G,S)$ if and only if
$\pi S\pi^{-1} = S$ holds.
\end{lemma}
\proof
(1) For any vertices $x,y \in G$ and element $\pi \in G$, we have
$$\varphi^\ell_\pi(x) \sim_L \varphi^\ell_\pi(y) \Leftrightarrow \exists s \in S~~  \varphi^\ell_\pi(y)^{-1}\varphi^\ell_\pi(x) = s \Leftrightarrow$$
$$\Leftrightarrow \exists s \in S~~  (\pi y)^{-1}\pi x = s
\Leftrightarrow \exists s \in S~~  y^{-1}x = s
\Leftrightarrow
x \sim_L y
$$

(2) Similar to item (1).

(3)  For any vertices $x,y \in G$ and element $\pi \in G$, we have $$\varphi^r_\pi(x) \sim_L \varphi^r_\pi(y) \Leftrightarrow \exists s \in S~~  \varphi^r_\pi(y)^{-1}\varphi^r_\pi(x) = s \Leftrightarrow$$
$$\Leftrightarrow \exists s \in S~~  (y\pi )^{-1}x\pi  = s
\Leftrightarrow \exists s \in S~~  y^{-1}x = \pi s\pi^{-1}
\Leftrightarrow
 y^{-1}x \in \pi S\pi^{-1}.
$$

(4) Similar to item (3). $\square$

\medskip
For any elements $\pi_\ell,\pi_r \in G$, denote by $\varphi_{\pi_\ell,\pi_r}$ the mapping that sends $x$ to $\pi_\ell x\pi_r$ for all $x \in G$.

\begin{lemma}\label{phipipi} For elements $\pi_\ell,\pi_r \in G$, the following statements hold.\\
{\rm (1)} $\varphi_{\pi_\ell,\pi_r}$ is the composition of $\varphi^\ell_{\pi_\ell}$ and $\varphi^r_{\pi_r}$.\\
{\rm (2)} $\varphi_{\pi_\ell,\pi_r}$ is an automorphism of $Cay_L(G,S)$ if and only if $\pi_r S\pi_r^{-1} = S$ holds.\\
{\rm (3)} $\varphi_{\pi_\ell,\pi_r}$ is an automorphism of $Cay_R(G,S)$ if and only if $\pi_\ell S\pi_\ell^{-1} = S$ holds.
\end{lemma}
\proof
(1) It follows from the definitions.

(2) It follows from Lemma \ref{aut}(3).

(3) It follows from Lemma \ref{aut}(4). $\square$

\section{Dual Seidel switching and Kronecker covering}
In this section we establish a connection between the phenomenon of multiple Kronecker covering and the dual Seidel switching operation. Let us note that, for any two non-isomorphic graphs $\Gamma_1$ and $\Gamma_2$, to share the bipartite double just means the bipartite double admits multiple Kronecker covering (at least, $\Gamma_1$ and $\Gamma_2$ are covered).

\begin{lemma}
    Let $\Gamma$ be the Kronecker cover of a graph $\Gamma'$.
    Then $\Spec(\Gamma) = \Spec(\Gamma') \cup -\Spec(\Gamma')$.
\end{lemma}

\begin{proof}
    It follows from $A(\Gamma) = A(\Gamma') \otimes \begin{bmatrix}0 & 1 \\1 & 0 \end{bmatrix}$.
\end{proof}

\begin{theorem}\label{CommonBD}
    Suppose a connected graph $\Gamma$ is the common bipartite double of $\Gamma_1$ and $\Gamma_2$, then there exists a permutation matrix $P$ such that $A(\Gamma_2) = PA(\Gamma_1)$ and $PA(\Gamma_1)P = A(\Gamma_1)$.
\end{theorem}

\begin{proof}
    Because $\Gamma_1$ and $\Gamma_2$ have the same bipartite double, the matrices
    $\begin{bmatrix} 0 & A(\Gamma_1) \\ A(\Gamma_1) & 0 \end{bmatrix}$ and
    $\begin{bmatrix} 0 & A(\Gamma_2) \\ A(\Gamma_2) & 0 \end{bmatrix}$
    are both an adjacency matrix of $\Gamma$. Since $\Gamma$ is connected, the bipartition is unique.
    This implies that $A(\Gamma_2)$ is obtained by permuting the rows and columns in $A(\Gamma_1)$, namely $A(\Gamma_2) = P A(\Gamma_1)Q$, where $P$ and $Q$ represent the row and the column permutations, respectively.
    In fact we may assume $A(\Gamma_2) = P A(\Gamma_1)$ because we can always reorder the vertices in the adjacency matrix.
    The symmetry of $A(\Gamma_1)$ yields
    $PA(\Gamma_1) = A^T(\Gamma_1) P^T = A(\Gamma_1) P^T$.
    In other words $PA(\Gamma_1)P = A(\Gamma_1)$. $\square$
\end{proof}

\begin{theorem}\label{sameKC}
    Let $\Gamma'$ be a graph obtained from a graph $\Gamma$ by dual Seidel switching.
    Then $\Gamma$ and $\Gamma'$ share the same bipartite double.
\end{theorem}

\begin{proof}
    Let $A$ be the adjacency matrix of $\Gamma$ and $P$ be the permutation matrix corresponding to the involution.
    The theorem follows from the following identity:
    \begin{equation}
        \begin{bmatrix}
            I & 0 \\
            0 & P
        \end{bmatrix}
        .
        \begin{bmatrix}
            0 & PA \\
            PA & 0
        \end{bmatrix}
        .
        \begin{bmatrix}
            I & 0 \\
            0 & P
        \end{bmatrix}^T
        =
        \begin{bmatrix}
            0 & A \\
            A & 0
        \end{bmatrix}.
    \end{equation}
    $\square$
\end{proof}

We have thus proved that all graphs that share the bipartite double are connected through dual Seidel switching, which gives a solution of \cite[Problem 1]{IP08}.

\section{Dual Seidel switching and Star graphs}
In this section, we apply dual Seidel switching to Star graphs, which, for every integer $n \geqslant 5$, gives a pair of two isomorphic connected integral graphs. In particular, we construct a new connected 4-regular integral graph in the case $n = 5$.

Let $n$ be a positive integer, $n \geqslant 3$. Consider the symmetric group $G = Sym_n$ and put $S = \{(1~i)~|~ i \in \{2,\ldots,n\}\}$.
The \emph{left Star graph} (resp. \emph{right Star graph}) is the Cayley graph $Cay_L(Sym_n,S)$ (resp. $Cay_R(Sym_n,S)$).

\begin{lemma}\label{stab}
For an element $\pi \in G$, the equality
$\pi S\pi^{-1} = S$ holds if and only if $\pi$ is an element from $Stab_G(1)$, where $Stab_G(1)$ is the stabilizer of 1 in $G$.
\end{lemma}
\begin{proof}
Let $\pi$ be a permutation from $Stab_G(1)$, which means that $\pi(1) = 1$. Note that, for any $i \in \{2,\ldots,n\}$, we have $\pi (1~i) \pi^{-1} = (1~j)$, where $j = \pi(i)$. Thus, $\pi S\pi^{-1} = S$ for all $\pi \in Stab_G(1)$.

Let $\pi$ be a permutation that does not belong to $Stab_G(1)$ and let $i$ be an element from $\{2,\ldots,n\}$ such that $\pi^{-1}(1) \ne i$. Then
the permutation $\pi(1~i)\pi^{-1}$ does stabilize
$1$ and thus does not belong to $S$, which means that
$\pi S\pi^{-1} \ne S$. $\square$
\end{proof}

Recall that, for any elements $\pi_\ell,\pi_r \in Sym_n$, $\varphi_{\pi_\ell,\pi_r}$ denotes the mapping that sends $x$ to $\pi_\ell x\pi_r$ for all $x \in Sym_n$.

\begin{lemma}\label{involution}
For  $\pi_\ell,\pi_r\in Sym_n$, if  $\pi_\ell,\pi_r$ satisfy\\
{\rm(1)} $\pi_\ell,\pi_r$ are both involutions;\\
{\rm(2)} $\pi_\ell,\pi_r$ have different parity;\\
{\rm(3)} $\pi_r S\pi_r^{-1} = S$;\\
{\rm(4)} $\pi_\ell$ is not conjugate to any element in $\pi_r S$,\\
then
$\varphi_{\pi_\ell,\pi_r}$ is an order 2 automorphism of the left Star graph $Cay_L(Sym_n,S)$ interchanging only non-adjacent vertices from different parts in bipartition of $Cay_L(Sym_n,S)$.
\end{lemma}
\begin{proof}
 For all vertices $x \in Sym_n$, by condition (2),
the vertices $x$ and $\pi_\ell~x\pi_r$ belong to different parts in  bipartition of $Cay_L(Sym_n,S)$.
In view of Lemma \ref{phipipi}(2), it suffices to show that   $x$ and $\pi_\ell~x\pi_r$ are non-adjacent in $Cay_L(Sym_n,S)$. Suppose there exists a vertex $x$ such that $x$ and $\pi_\ell~x\pi_r$ are adjacent, which means
that $\pi_r~x^{-1}\pi_\ell~x = (1~i)$ for some $i \in \{2,\ldots,n\}$. Thus we have $x^{-1}\pi_\ell~x = \pi_r(1~i)$, a contradiction with condition (4). $\square$
\end{proof}

\medskip
In Corollary \ref{inv1Star} and Corollary \ref{invStar} we consider specific mappings $\varphi_{\pi_\ell,\pi_r}$, which, in view of Lemma \ref{involution}, turn out to be order 2 automorphisms of the left Star graph $Cay_L(Sym_n,S)$ interchanging only non-adjacent vertices from different parts of the bipartition of $Cay_L(Sym_n,S)$.

\begin{corollary}\label{inv1Star}
For a positive integer $n \geqslant 4$,
$\varphi_{id,(2~3)}$ is an order 2 automorphism of the left Star graph $Cay_L(Sym_n,S)$ interchanging only non-adjacent vertices, where $id \in Sym_n$ is the identity.
\end{corollary}
\begin{proof} Straightforward. $\square$
\end{proof}

\begin{corollary}\label{invStar}
For a positive integer $n \geqslant 5$,
$\varphi_{(2~4),(2~3)(4~5)}$ is an order 2 automorphism of the left Star graph $Cay_L(Sym_n,S)$ interchanging only non-adjacent vertices.
\end{corollary}
\begin{proof}
It suffices to show that, for all  $x \in Sym_n$ and $i \in \{2,\ldots, n\}$,
that $x^{-1}(2~4)x \neq (2~3)(4~5)(1~i)$. Note that, $(2~3)(4~5)(1~i)$ is either a product of a transposition and a cycle of length 3 (the case $2 \leqslant i  \leqslant 5$) or a product of three disjoint transpositions (the case $i \geqslant 6$). Since the conjugation preserves the cyclic structure of a permutation, we are completed. $\square$
\end{proof}

\begin{theorem}\label{TStar}
Let  $\varphi$ be an automorphism of left Star graph $Cay_L(Sym_n, S)$ satisfying the four conditions in Lemma~\ref{involution}. Then the graph obtained from  $Cay_L(Sym_n, S)$ by dual Seidel switching induced by $\varphi$ consists of two isomorphic connected components. Moreover, the bipartite double of such a component is isomorphic to $Cay_L(Sym_n, S)$.
\end{theorem}
\begin{proof}
We denote by $A$ the adjacency matrix of $Cay_L(Sym_n, S)$.
Since $Cay_L(Sym_n, S)$ is bipartite,  we can relabel the vertices such that its adjacency matrix $A$  has the following form
$$A=\left(
  \begin{array}{cc}
   0 & B \\
   B^T & 0\\
  \end{array}
\right).$$

 Let $P$ denote the permutation matrix corresponding to the automorphism $\varphi$ of $Cay_L(Sym_n, S)$.  Since $\varphi$ just swaps the parts of the bipartition, accordingly $P$ is of the form as follows
$$P=\left(
  \begin{array}{cc}
   0 & Q^T \\
   Q & 0\\
  \end{array}
\right).$$
  Since $PAP=A$, we have $B^T=QBQ$ and then
  $$PA=\left(
  \begin{array}{cc}
   Q^TB^T & 0 \\
   0& QB \\
  \end{array}
\right)=\left(
  \begin{array}{cc}
    BQ & 0 \\
   0& QB \\
  \end{array}
\right).$$
Thus, the graph (denoted by $H$) obtained from $Cay_L(Sym_n, S)$ by dual Seidel switching induced by $\varphi$ has at least two connected components.  Next we can show that $H$ has at most two connected components.
By Theorem~\ref{sameKC},  $H$ and  $Cay_L(Sym_n, S)$ have the same bipartite double and by Lemma~\ref{connectedness},  the bipartite double of $Cay_L(Sym_n, S)$ consists of two copies $Cay_L(Sym_n, S)$.  Thus by Lemma~\ref{connectedness}, $H$ cannot have more than two connected components.
Further,  as $Q^T(QB)Q=BQ$, namely $QB$ and $BQ$ are permutation similar,  the  graph $H$ obtained by applying
dual Seidel switching on $Cay_L(Sym_n, S)$
has two isomorphic connected components.

The fact that the bipartite double of every of the two connected components is isomorphic to $Cay_L(Sym_n, S)$ follows from the property that the combinatorial polarity $\varphi$ interchanges the parts of the bipartition of $Cay_L(Sym_n, S)$.
$\square$
\end{proof}

\medskip

Thus, Theorem \ref{TStar} with Corollary \ref{inv1Star} and Corollary \ref{invStar}, for every integer $n \geqslant 5$,  give two connected integral graphs with $n!/2$ vertices whose bipartite double is isomorphic to the left Star graph $Cay_L(Sym_n,S)$ (in other words, $Cay_L(Sym_n,S)$ is a Kronecker cover of this new connected integral graph).

\begin{proposition}\label{prop1}
    For an integer $n \geqslant 4$, one of the two connected components of the resulting integral graph given by Theorem \ref{TStar} and Corollary \ref{inv1Star} is isomorphic to $Cay_L(Alt_n,S_1)$, where $S_1 = (2~3)S$ and $S = \{(1~2),(1~3), \ldots, (1~n)\}$.
\end{proposition}
\begin{proof}
    Let us investigate the adjacency relation in the resulting graph. By the definition of the left Star graph, for a vertex $x$, the neighbourhood of $x$ is given by $N(x) = xS$. Denote by $N'(x)$ the neighbourhood of $x$ in the resulting graph. By Lemma \ref{Images} and Corollary \ref{inv1Star}, we have $$N'(x) = N(\varphi(x)) = \varphi(x)S = x(23)S,$$
    which is equivalent that the resulting graph is a left Cayley graph with generating set $S_1 = (2~3)S$. It is easy to see that $S_1$ contains only even permutations. Thus, the two connected components of the resulting graph are induced by the sets of even and odd permutations of $Sym_n$. $\square$
\end{proof}

\medskip
Proposition \ref{prop1}, Theorem \ref{TStar} and Corollary \ref{inv1Star} give a vertex-transitive 4-regular graph, which is already known from \cite[page 402]{MW15} as the Cayley graph $H_{15}$; it has spectrum $\{ (-3)^4, (-2)^{17}, (-1)^{4}, 0^{15}, 2^{11}, 3^8, 4^1 \}$.
Theorem \ref{TStar} and Corollary \ref{invStar} give a new 4-regular graph, which is not vertex-transitive and has spectrum $\{ (-3)^{7}, (-2)^{13}, (-1)^3, 0^{15},  1^1, 2^{15}, 3^{5}, 4^1 \}$.

\begin{problem}\label{ProblemStar}
What are integral graphs whose bipartite double is isomorphic to the left Star graph $Cay_L(Sym_n,S)$?
\end{problem}

We point out that to solve Problem \ref{ProblemStar}, in view of Theorem \ref{CommonBD} and Theorem \ref{sameKC}, it suffices to apply dual Seidel switching to the family of connected integral graphs given by Theorem \ref{TStar} and Corollary \ref{inv1Star}. Moreover, it is guaranteed that the family of connected integral given by Theorem \ref{TStar} and Corollary \ref{invStar} can be obtained for $Cay_L(Alt_n,S_1)$ by dual Seidel switching.

Let us note that Problem \ref{ProblemStar} is a particular case of \cite[Problem 1]{IP08}. In this section we have provided more motivation to investigate Problem 1 as well as \cite[Problem 1]{IP08}.

We conclude this section with one more open problem.
With use of computer, we checked that in the cases
$n = 5$ and $n = 6$ the integral graphs given by Theorem \ref{TStar} and Corollary \ref{invStar} are not vertex-transitive.

\begin{problem}
Show that integral graphs given by Theorem \ref{TStar} and Corollary \ref{invStar} are not vertex-transitive for $n \geqslant 5$.
\end{problem}

\section{Dual Seidel switching and Odd graphs}
In this section, we apply dual Seidel switching to the Odd graphs and construct an infinite family of integral graphs. In particular, we construct two new 4-regular integral graphs.
\subsection{Odd graphs}
For a positive integer $m$, the \emph{Odd graph}, denoted by $O_{m+1}$, is the graph whose vertex set is the set of $m$-subsets of a $(2m + 1)$-set $X$, where two $m$-sets are adjacent if and only if they are disjoint.
It is easy to see that every permutation of $X$ induces an automorphism of $O_{m+1}$. The following lemma gives the complete information about the spectrum and the automorphism group of the graph $O_{m+1}$.

\begin{lemma}[{\cite[Proposition 9.1.7]{BCN89}}]\label{OddGraph} The following statements hold.\\
{\rm (1)} The eigenvalues of $O_{m+1}$ are $(-1)^i(m+1-i)$ with multiplicity $\binom{2m+1}{i} - \binom{2m+1}{i-1}$, where $i$ runs over $\{0, 1, \ldots, m\}$;\\
{\rm (2)} The automorphism group of $O_{m+1}$ consists of the automorphisms induced by the permutations of $X$, which is the symmetric group $Sym_{2m+1}$.
\end{lemma}

The Odd graphs are among a more general graph family called Johnson graphs.
Let $X$ be a set of size $n$ and let $V = \binom{X}{d}$ be the collection of all $d$-subsets of $X$.
We denote by $J(n,d,r)$ the graph whose vertex set is $X$, where two vertices $A$ and $B$ are adjacent if and only if $|A \cap B| = d - r$.
The Odd graph $O_{m+1}$ is exactly the graph $J(2m+1,m,m)$.
The following lemma, originally stated for Johnson graphs, describes the eigenspaces of the Odd graph.

\begin{lemma}[{\cite[Lemma 6.2.3 and Theorem 6.3.3]{CG16}}] \label{lem:eigenspace}
    Let $A$ be the adjacency matrix of $O_{m+1}$ indexed by the $m$-subsets.
    For a subset $T \subseteq \{1,2, \ldots, 2m+1\}$, let $V_T$ be the collection of $m$-subsets that contain $T$.
    Let $W_i$ be the vector space spanned by the characteristic vectors $\mathds{1}_{V_T}$, where $T$ runs over all $i$-subsets.
    Then $W_{i-1} \subset W_i$ and the $i$-th eigenspace of $A$ is $W_i \cap W_{i-1}^\perp$, whose dimension is $\binom{2m+1}{i} - \binom{2m+1}{i-1}$, where $i=0,1, \ldots, m$ and $W_{-1}$ is regarded as the $0$-dimensional vector space.
\end{lemma}

\subsection{Dual Seidel switching and Odd graphs}

For a positive integer $t$, put $\tau_t:=(1~2)\ldots(2t-1~ 2t)$.
Recall that the vertex set $V$ of the Odd graph consists of $m$-subsets of an $(2m+1)$-set $X$. The permutation $\tau_t$ naturally acts on $X$. Denote by $\varphi_t$ the corresponding action on $V$.

\begin{lemma}
Given a positive integer $m$, $m \geqslant 2$, the following statements hold.
{\rm (1)} For any $t \in \{1,\ldots,m-1\}$, the permutation $\tau_t$ induces an involution $\varphi_t$ of $O_{m+1}$ that interchanges only non-adjacent vertices;\\
{\rm (2)} The permutation $\tau_m$ induces an involution $\varphi_m$ of $O_{m+1}$ that interchanges adjacent vertices as well as non-adjacent vertices.
\end{lemma}
\begin{proof}
(1) The image of any vertex $Y_1 = \{s_1, s_2, \ldots, s_m\}$ under the involution $\varphi_t$ is the vertex $Y_2 = \{\tau_t(s_1), \tau_t(s_2), \ldots, \tau_t(s_m)\}$.
If they are adjacent, then the two $m$-subsets are disjoint.
We must have $\tau_t(s_i) \not\in Y_1$ for every $1 \leqslant i \leqslant m$.
So no element of $Y_1$ is fixed by $\tau_t$ and $s_i, s_j$ cannot be in the same $2$-cycle of $\tau_t$ for $1 \leqslant i \neq j \leqslant m$.
This forces $t \geqslant m$. Contradiction.

(2) The involution $\varphi_m$ interchanges the vertices $\{1,3,\ldots, 2m-1\}$ and $\{2,4,\ldots, 2m\}$, which are adjacent.
$\square$
\end{proof}

\medskip
It follows from Lemma \ref{OddGraph}(1) that the smallest eigenvalue of $O_{m+1}$ is $-m$ with multiplicity $2m$, and $O_{m+1}$ has no eigenvalue $m$.

For any positive integers $i,j$, $1 \leqslant i,j \leqslant 2m+1, i\ne j$, let us introduce a partition of the vertex set of $O_{m+1}$ into $V_{i,j},V_{i,\overline{j}},V_{\overline{i},j},V_{\overline{i},\overline{j}}$, where
$$V_{i,j}:=\{m\text{-subsets of }X \text{ that contain both $i$ and $j$}\},$$
$$V_{i,\overline{j}}:=\{m\text{-subsets of }X \text{ that contain $i$ and do not contain $j$}\},$$
$$V_{\overline{i},j}:=\{m\text{-subsets of }X \text{ that do not contain $i$ and contain $j$}\},$$
$$V_{\overline{i},\overline{j}}:=\{m\text{-subsets of }X \text{ that do not contain $i$ or $j$}\},$$
and define a function $f_{i,j}:  \,\, V(O_{m+1})\rightarrow \mathbb{R}$ by the following rule. For any $m$-subset $Y$, put
$$
f_{i,j}(Y):=
\left\{
  \begin{array}{ll}
    1, & Y \in V_{i,\overline{j}}; \\
    -1, & Y \in V_{\overline{i},j};\\
    0, & Y \in V_{i,j} \cup V_{\overline{i},\overline{j}}.
  \end{array}
\right.
$$

For any vertex subset $W \subseteq V$, we denote by $\mathds{1}_W$ the characteristic function of $W$, namely
$$
\mathds{1}_W(v) =
\begin{cases}
    1, & \textif v \in W; \\
    0, & \otherwise.
\end{cases}
$$

One can see that we have $f_{i,j} = \mathds{1}_{V_{i,\overline{j}}} - \mathds{1}_{V_{\overline{i},j}}$.

\begin{lemma}\label{partition}
For any integers $m, i,j$, where $m\geqslant 2$ and $1 \leqslant i,j \leqslant m, i\ne j$, the following statements hold.\\
{\rm (1)} The partition of the vertex set of $O_{m+1}$ into $V_{i,j},V_{i,\overline{j}},V_{\overline{i},j},V_{\overline{i},\overline{j}}$ is equitable with quotient matrix
$\left(
  \begin{array}{cccc}
    0 & 0 & 0 & m+1 \\
    0 & 0 & m & 1 \\
    0 & m & 0 & 1 \\
    m-1 & 1 & 1 & 0  \\
  \end{array}
\right)$;\\
{\rm (2)} The function $f_{i,j}$ is an $-m$-eigenfunction of $O_{m+1}$.
\end{lemma}
\begin{proof}
(1) Straightforward;\\
(2) It follows from item (1). $\square$
\end{proof}

\begin{theorem}\label{basis}
Eigenfunctions $f_{1,2m+1}, f_{2,2m+1}, \ldots, f_{2m,2m+1}$ form a basis of the $-m$-eigenspace of $O_{m+1}$.
\end{theorem}
\begin{proof}
We regard the eigenfunctions as vectors in the space $\mathbb{R}^{V(O_{m+1})}$.
Let us consider the Gram matrix of these vectors.
Since $\langle f_{i,2m+1},f_{i,2m+1} \rangle = \langle \mathds{1}_{V_{i,\overline{2m+1}}} - \mathds{1}_{V_{\overline{i},2m+1}},
\mathds{1}_{V_{i,\overline{2m+1}}} - \mathds{1}_{V_{\overline{i},2m+1}}
\rangle
=
\left| V_{i,\overline{2m+1}}\right| + \left| V_{\overline{i},2m+1} \right|
=
2 \binom{2m-1}{m-1}
$ for $i=1,2, \ldots, 2m$ and

\begin{align*}
     & \langle f_{i,2m+1},f_{j,2m+1} \rangle \\
    = & \langle \mathds{1}_{V_{i,\overline{2m+1}}} - \mathds{1}_{V_{\overline{i},2m+1}},
\mathds{1}_{V_{j,\overline{2m+1}}} - \mathds{1}_{V_{\overline{j},2m+1}}
\rangle \\
    = & \left| V_{i,\overline{2m+1}} \cap V_{j,\overline{2m+1}}\right| + \left| V_{\overline{i},2m+1} \cap V_{\overline{j}, 2m+1}\right| \\
    = & \binom{2m-2}{m-2} + \binom{2m-2}{m-1}\\
    =& \binom{2m-1}{m-1}
\end{align*}
for $1 \leqslant i \neq j \leqslant 2m$.
Therefore the Gram matrix $G = \binom{2m-1}{m-1}(J+I)$, where $J$ is the all-one matrix.
So $G$ is non-singular.
Hence the $2m$ eigenfunctions $f_{1,2m+1}, f_{2,2m+1}, \ldots, f_{2m,2m+1}$ are linearly independent.
By Lemma \ref{OddGraph}(1), they form a basis of the $-m$-eigenspace of $O_{m+1}$.
$\square$
\end{proof}

\medskip

Given positive integers $m \geqslant 2$ and $1 \leqslant t \leqslant m-1$, we denote by $O^t_{m+1}$  the graph obtained from $O_{m+1}$ by dual Seidel switching w.r.t. the involution $\varphi_t$ of $O_{m+1}$ induced by the permutation $\tau_t$.

\begin{theorem}\label{TOt} Let $m$ be a positive integer, $m \geqslant 2$. Then the following statements hold.\\
{\rm (1)} For any integer $t$, $1 \leqslant t \leqslant m-1$, the graph $O^t_{m+1}$ has eigenvalue $m$ with multiplicity $t$.\\
{\rm (2)} The $m-1$ graphs $O^t_{m+1}$, where $1 \leqslant t \leqslant m-1$, are integral and pairwise non-isomorphic.
\end{theorem}
\begin{proof}
(1) For a real square matrix $A$, the spectrum of $A^2$ is determined by the spectrum of $A$ by squaring the eigenvalues and summing up the multiplicity of opposite eigenvalues of $A$.
Since $\varphi_t$ is an involution, the adjacency matrices of $O_{m+1}$ and $O_{m+1}^t$ share the same square.
We will show (1) by determining the $-m$-eigenspace and $m$-eigenspace of $O_{m+1}^t$.
Note that $m$ is not an eigenvalue of $O_{m+1}$ and the multiplicity of $-m$ of $O_{m+1}$ is $2m$.
We study the action of $\varphi_t$ on the eigenfunctions in Theorem \ref{basis}.
One can see that $\varphi_t V_{i,j} = V_{\tau_t(i),\tau_t(j)}$ (and similarly for $V_{i,\overline{j}}$, $V_{\overline{i},j}$ and $V_{\overline{i},\overline{j}}$).
So we have $\varphi_t f_{i,2m+1} = f_{\tau_t(i),2m+1}$ for $1 \leqslant t \leqslant m-1$.
Let $A$ be the adjacency matrix of $O_{m+1}$ and $B$ the adjacency matrix of $O_{m+1}^t$.
They are related by $B = P A$ where $P = P^T$ is the permutation matrix of the involution $\varphi_t$.

For $1 \leqslant i \leqslant t$, we have $B (f_{2i-1,2m+1} - f_{2i,2m+1}) = PA (f_{2i-1,2m+1} - f_{2i,2m+1}) = -m P (f_{2i-1,2m+1} - f_{2i,2m+1}) = m (f_{2i-1,2m+1} - f_{2i,2m+1})$ and $B (f_{2i-1,2m+1} + f_{2i,2m+1}) = -m (f_{2i-1,2m+1} + f_{2i,2m+1})$, which shows that $f_{2i-1,2m+1} - f_{2i,2m+1}$ and $f_{2i-1,2m+1} + f_{2i,2m+1}$ are an $m$-eigenfunction and a $(-m)$-eigenfunction of $O_{m+1}^t$, respectively.

For $t < i \leqslant m-1$, we have $B f_{i,2m+1} = -m f_{i,2m+1}$, which shows that $f_{i,2m+1}$ remains a $(-m)$-eigenfunction of $O_{m+1}^t$.

Thus, we have constructed $t$ eigenfunctions of the eigenvalue $m$ and $(2m-t)$ eigenfunctions of the eigenvalue $-m$ for the graph $O_{m+1}^t$.
It is clear that these eigenfunctions are all linearly independent.
Since $t + (2m-t) = 2m$, the $m$-eigenspace and $-m$-eigenspace of $O_{m+1}^t$ are determined.

(2) It follows directly from (1). $\square$
\end{proof}

\medskip
In the case $m = 2$, the Odd graph $O_{3}$ is isomorphic to the Petersen graph. Theorem \ref{TOt}, applied to $O_3$, gives a cubic integral graph with spectrum $\{ (-2)^3, (-1)^2, 1^3, 2^1, 3^1 \}$.
In the case $m =3$, Theorem \ref{TOt} gives two new 4-regular graphs with spectra $\{(-3)^5, (-2)^4, (-1)^9, 1^5, 2^{10}, 3^1, 4^1\}$ and $\{ (-3)^4, (-2)^6, (-1)^8, 1^6, 2^8, 3^2, 4^1 \}$. We point out that these three graphs are not vertex-transitive.

The following theorem determines the spectrum of the integral graphs found in Theorem \ref{TOt}.

\begin{theorem}\label{SpecOt}
    The spectrum of $O_{m+1}^t$ is determined as follows. For $i = 0, \ldots, m$,
    the eigenvalue $(-1)^{i+1} (m+1-i)$ is of multiplicity $a_{i}$ and the eigenvalue $(-1)^{i} (m+1-i)$ is of multiplicity $\binom{2m+1}{i} - \binom{2m+1}{i-1} - a_i$, where $$a_i = \frac{|\{i\text{-subsets not fixed by } \tau_t\}| - |\{(i-1)\text{-subsets not fixed by } \tau_t\}|}{2};$$ here we mean that $0$-subset is the empty set, hence fixed by every permutation, and $-1$-subset does not exist, hence the cardinality is zero.
\end{theorem}

\begin{proof}
    We adopt the notations in Lemma \ref{lem:eigenspace}.
    By Lemma \ref{OddGraph} and Lemma \ref{lem:eigenspace}, the $(-1)^{i} (m+1-i)$-eigenspace of $O_{m+1}$ is  $U_i := W_i \cap W_{i-1}^\perp$.
    Note that $v = \frac{v + \varphi_t v}{2} + \frac{v - \varphi_t v}{2}$ gives a decomposition $\mathbb{R}^{\binom{X}{m}} = F \oplus N$, where $\varphi_t f=f$ for every $f \in F$ and $\varphi_t n = -n$ for every $n \in N$.
    Let $F_i = U_i \cap F$ and $N_i = U_i \cap N$.
    Then $N_i$ is the $(-1)^{i+1} (m+1-i)$-eigenspace of $O_{m+1}^t$ and $F_i$ is the $(-1)^{i} (m+1-i)$-eigenspace of $O_{m+1}^t$.
    To determine the dimension of $N_i$, we consider $WN_i = W_i \cap N$.
    Then $N_i = WN_i \cap U_i$ and $WN_i = N_0 \oplus N_1 \oplus \cdots \oplus N_i$.
    So $\dim N_i = \dim WN_i - \dim WN_{i-1}$.
    Note that $\dim WN_i$ can be directly computed by considering the action of $\varphi_t$ on the basis of $W_i$. \qed
\end{proof}

\section{Concluding remarks}

Note that Theorem \ref{TOt} exhaust all involutions of the Odd graphs that interchange only non-adjacent vertices, while Theorem \ref{TStar} with Corollary \ref{inv1Star} and Corollary \ref{invStar} give examples of two such involutions of the Star graph.
The bipartite double of the Odd graph is known as the doubled Odd graph (see \cite[p. 259]{BCN89}).
In this paper, in view of Theorem \ref{CommonBD} and Theorem \ref{sameKC}, we have actually found all the graphs whose bipartite double is isomorphic to the doubled Odd graph (in other words, whose Kronecker cover is isomorphic to the doubled Odd graph). Let us mention here the paper \cite{Dam06}, which is closely related to this paper and where graphs with the same spectrum as a distance-regular graph were studied. In particular, the authors of \cite{Dam06} found an infinite family of graphs with the same spectrum of a doubled Odd graph.

We are wondering if the dual Seidel switching can be fruitfully applied to other known 4-regular graphs.
Let us mention that the full automoprhism group of the Star graph is known from \cite{F06}.

Finally, we point out that Theorem \ref{basis} gives a basis for the $-m$-eigenspace of the Odd graph $O_{m+1}$, which is of interest by itself.

\section*{Acknowledgment} \label{Ack}
The reported study was funded by RFBR according to the research project 20-51-53023.
 Sergey Goryainov, Elena Konstantinova and Honghai Li thank TGMRC (Three Gorges Mathematical Research Center) of China Three Gorges University in Yichang, Hubei, China, for supporting the visits of the authors to work on this research project in April 2019. Honghai Li is supported by the NSF of China (No. 11561032), the NSF for Jiangxi Distinguished Young Scholars (No. 20171BCB23032), and  the NSF of Jiangxi Province (No. 20192BAB201001).
  Da Zhao is supported by the NSF of China (No. 11671258). Sergey Goryainov is supported by STCSM (No. 17690740800).
 The work is supported by Mathematical Center in Akademgorodok, the
agreement with Ministry of Science and High Education of the Russian
Federation number 075-15-2019-1613.
The authors thank Professor Vladislav Kabanov who pointed out that dual Seidel switching applied to an integral graph produces an integral graph. Our thanks goes to Professor Jack Koolen for his comments on the paper. Finally, the authors thank the anonymous referee whose suggestions significantly improved the paper.

\end{document}